\documentclass[11pt]{amsart}

\usepackage[all]{xy}
\usepackage{mathrsfs}
\usepackage{amsmath, amsthm, amssymb, amscd, amsgen,amsxtra,amsfonts, amsbsy}
\allowdisplaybreaks
\usepackage[all]{xy}
\usepackage{verbatim}

\usepackage{tikz}
\usepackage{pgfplots}
\usepackage{tikz-cd}
\usepackage{etex}
\usepackage{calligra,mathrsfs}
\usepackage{color}
\definecolor{shadecolor}{gray}{0.875}
\definecolor{dblue}{rgb}{0,0,.6}
\usepackage[colorlinks=true, linkcolor=dblue, citecolor=dblue, filecolor = dblue, menucolor = dblue, urlcolor = dblue]{hyperref}
\usepackage{mathtools}
\usepackage{extarrows}
\usepackage{ifthen}
\usepackage{bm}

\usepackage{graphicx}
\usepackage{caption,rotating}
\usepackage{color}
\usepackage{subcaption}

\usepackage{enumitem}

\newcommand{\mathds}[1]{{\mathbb #1}}

\numberwithin{equation}{section}

\begin{document}
%
%
%
\theoremstyle{definition}
\newtheorem{Definition}{Definition}[section]
\newtheorem*{Definitionx}{Definition}
\newtheorem{Convention}{Definition}[section]
\newtheorem{Construction}[Definition]{Construction}
\newtheorem{Example}[Definition]{Example}
\newtheorem{Exercise}[Definition]{Exercise}
\newtheorem{Examples}[Definition]{Examples}
\newtheorem{Remark}[Definition]{Remark}
\newtheorem*{Remarkx}{Remark}
\newtheorem{Remarks}[Definition]{Remarks}
\newtheorem{Caution}[Definition]{Caution}
\newtheorem{Conjecture}[Definition]{Conjecture}
\newtheorem*{Conjecturex}{Conjecture}
\newtheorem{Question}[Definition]{Question}
\newtheorem*{Questionx}{Question}
\newtheorem*{Acknowledgements}{Acknowledgements}
\newtheorem*{Notation}{Notation}
\newtheorem*{Organization}{Organization}
\newtheorem*{Disclaimer}{Disclaimer}
\theoremstyle{plain}
\newtheorem{Theorem}[Definition]{Theorem}
\newtheorem*{Theoremx}{Theorem}

\newtheorem{Proposition}[Definition]{Proposition}
\newtheorem*{Propositionx}{Proposition}
\newtheorem{Lemma}[Definition]{Lemma}
\newtheorem{Corollary}[Definition]{Corollary}
\newtheorem*{Corollaryx}{Corollary}
\newtheorem{Fact}[Definition]{Fact}
\newtheorem{Facts}[Definition]{Facts}
\newtheoremstyle{voiditstyle}{3pt}{3pt}{\itshape}{\parindent}%
{\bfseries}{.}{ }{\thmnote{#3}}%
\theoremstyle{voiditstyle}
\newtheorem*{VoidItalic}{}
\newtheoremstyle{voidromstyle}{3pt}{3pt}{\rm}{\parindent}%
{\bfseries}{.}{ }{\thmnote{#3}}%
\theoremstyle{voidromstyle}
\newtheorem*{VoidRoman}{}

\newenvironment{specialproof}[1][\proofname]{\noindent\textit{#1.} }{\qed\medskip}
\newcommand{\blowup}{\rule[-3mm]{0mm}{0mm}}
\newcommand{\cal}{\mathcal}
\newcommand{\Aff}{{\mathds{A}}}
\newcommand{\BB}{{\mathds{B}}}
\newcommand{\CC}{{\mathds{C}}}
\newcommand{\EE}{{\mathds{E}}}
\newcommand{\FF}{{\mathds{F}}}
\newcommand{\GG}{{\mathds{G}}}
\newcommand{\HH}{{\mathds{H}}}
\newcommand{\NN}{{\mathds{N}}}
\newcommand{\ZZ}{{\mathds{Z}}}
\newcommand{\PP}{{\mathds{P}}}
\newcommand{\QQ}{{\mathds{Q}}}
\newcommand{\RR}{{\mathds{R}}}
\newcommand{\BA}{{\mathds{A}}}
\newcommand{\Liea}{{\mathfrak a}}
\newcommand{\Lieb}{{\mathfrak b}}
\newcommand{\Lieg}{{\mathfrak g}}
\newcommand{\Liem}{{\mathfrak m}}
\newcommand{\ideala}{{\mathfrak a}}
\newcommand{\idealb}{{\mathfrak b}}
\newcommand{\idealg}{{\mathfrak g}}
\newcommand{\idealm}{{\mathfrak m}}
\newcommand{\idealp}{{\mathfrak p}}
\newcommand{\idealq}{{\mathfrak q}}
\newcommand{\idealI}{{\cal I}}
\newcommand{\lin}{\sim}
\newcommand{\num}{\equiv}
\newcommand{\dual}{\ast}
\newcommand{\iso}{\cong}
\newcommand{\homeo}{\approx}
\newcommand{\mm}{{\mathfrak m}}
\newcommand{\pp}{{\mathfrak p}}
\newcommand{\qq}{{\mathfrak q}}
\newcommand{\rr}{{\mathfrak r}}
\newcommand{\pP}{{\mathfrak P}}
\newcommand{\qQ}{{\mathfrak Q}}
\newcommand{\rR}{{\mathfrak R}}
\newcommand{\OO}{{\cal O}}
\newcommand{\numero}{{n$^{\rm o}\:$}}
\newcommand{\mf}[1]{\mathfrak{#1}}
\newcommand{\mc}[1]{\mathcal{#1}}
\newcommand{\into}{{\hookrightarrow}}
\newcommand{\onto}{{\twoheadrightarrow}}
\newcommand{\Spec}{{\rm Spec}\:}
\newcommand{\BigSpec}{{\rm\bf Spec}\:}
\newcommand{\Spf}{{\rm Spf}\:}
\newcommand{\Proj}{{\rm Proj}\:}
\newcommand{\Pic}{{\rm Pic }}
\newcommand{\MW}{{\rm MW }}
\newcommand{\Br}{{\rm Br}}
\newcommand{\NS}{{\rm NS}}
\newcommand{\Sym}{{\mathfrak S}}
\newcommand{\Aut}{{\rm Aut}}
\newcommand{\Autp}{{\rm Aut}^p}
\newcommand{\ord}{{\rm ord}}
\newcommand{\coker}{{\rm coker}\,}
\newcommand{\divisor}{{\rm div}}
\newcommand{\Def}{{\rm Def}}
\newcommand{\rank}{\mathop{\mathrm{rank}}\nolimits}
\newcommand{\Ext}{\mathop{\mathrm{Ext}}\nolimits}
\newcommand{\EXT}{\mathop{\mathscr{E}{\kern -2pt {xt}}}\nolimits}
\newcommand{\Hom}{\mathop{\mathrm{Hom}}\nolimits}
\newcommand{\Bk}{\mathop{\mathrm{Bk}}\nolimits}
\newcommand{\HOM}{\mathop{\mathscr{H}{\kern -3pt {om}}}\nolimits}
\newcommand{\calA}{\mathscr{A}}
\newcommand{\calC}{\mathscr{C}}
\newcommand{\calH}{\mathscr{H}}
\newcommand{\calL}{\mathscr{L}}
\newcommand{\calM}{\mathscr{M}}
\newcommand{\calN}{\mathscr{N}}
\newcommand{\calX}{\mathscr{X}}
\newcommand{\calK}{\mathscr{K}}
\newcommand{\calD}{\mathscr{D}}
\newcommand{\calY}{\mathscr{Y}}
\newcommand{\calF}{\mathscr{F}}
\newcommand{\calG}{\mathscr{G}}
\newcommand{\calE}{\mathscr{E}}
\newcommand{\CN}{\mathcal{N}}
\newcommand{\DD}{\mathcal{D}}
\newcommand{\CCC}{\mathcal{C}}
\newcommand{\CK}{\mathcal{K}}
\newcommand{\CV}{\mathcal{V}}

\newcommand{\chari}{\mathop{\mathrm{char}}\nolimits}
\newcommand{\ch}{\mathop{\mathrm{ch}}\nolimits}
\newcommand{\CH}{\mathop{\mathrm{CH}}\nolimits}
\newcommand{\supp}{\mathop{\mathrm{supp}}\nolimits}
\newcommand{\codim}{\mathop{\mathrm{codim}}\nolimits}
\newcommand{\td}{\mathop{\mathrm{td}}\nolimits}
\newcommand{\Span}{\mathop{\mathrm{Span}}\nolimits}
\newcommand{\Gal}{\mathop{\mathrm{Gal}}\nolimits}
\newcommand{\sym}{\mathop{\mathrm{Sym}}\nolimits}
\newcommand{\cl}{\mathop{\mathrm{cl}}\nolimits}
\newcommand{\RCH}{\mathop{\mathrm{RCH}}\nolimits}
\newcommand{\Cr}{\mathop{\mathrm{Cr}}\nolimits}
\newcommand{\piet}{{\pi_1^{\rm \acute{e}t}}}
\newcommand{\Het}[1]{{H_{\rm \acute{e}t}^{{#1}}}}
\newcommand{\Hfl}[1]{{H_{\rm fl}^{{#1}}}}
\newcommand{\Hcris}[1]{{H_{\rm cris}^{{#1}}}}
\newcommand{\HdR}[1]{{H_{\rm dR}^{{#1}}}}
\newcommand{\hdR}[1]{{h_{\rm dR}^{{#1}}}}
\newcommand{\loc}{{\rm loc}}
\newcommand{\et}{{\rm \acute{e}t}}
\newcommand{\defin}[1]{{\bf #1}}

\ifthenelse{\equal{1}{1}}{
\ifthenelse{\equal{2}{2}}{
\newcommand{\blue}[1]{{\color{blue}#1}}
\newcommand{\green}[1]{{\color{green}#1}}
\newcommand{\red}[1]{{\color{red}#1}}
\newcommand{\cyan}[1]{{\color{cyan}#1}}
\newcommand{\magenta}[1]{{\color{magenta}#1}}
\newcommand{\yellow}[1]{{\color{yellow}#1}} 
}{
\newcommand{\blue}[1]{#1}
\newcommand{\green}[1]{#1}
\newcommand{\red}[1]{#1}
\newcommand{\cyan}[1]{#1}
\newcommand{\magenta}[1]{#1}
\newcommand{\yellow}[1]{#1} 
}
}{
\newcommand{\blue}[1]{}
\newcommand{\green}[1]{}
\newcommand{\red}[1]{}
\newcommand{\cyan}[1]{}
\newcommand{\magenta}[1]{}
\newcommand{\yellow}[1]{} 
}

\renewcommand{\HH}{{\rm{H}}}

\title[Severi Varieties on Ruled Surfaces]{Severi Varieties on Ruled Surfaces over Elliptic Curves}

\date{May 17, 2024}

\author{Xiaotian Chang}
\address{632 Central Academic Building\\ University of Alberta\\ Edmonton, Alberta T6G 2G1, CANADA}
\email{xchang@ualberta.ca}

\author{Xi Chen}
\address{632 Central Academic Building\\
University of Alberta\\
Edmonton, Alberta T6G 2G1, CANADA}
\email{xichen@math.ualberta.ca}

\author{Adrian Zahariuc}
\address{Department of Mathematics and Statistics\\
University of Windsor\\
401 Sunset Avenue\\
Windsor, Ontario N9B 3P4}
\email{Adrian.Zahariuc@uwindsor.ca}


\keywords{Severi Variety, Deformation Theory}

\subjclass{14C25, 14C30, 14C35}

\begin{abstract}
We proved that the general members of Severi varieties on an Atiyah ruled surface over a general elliptic curve have nodes and ordinary triple points as singularities.
\end{abstract}

\maketitle

\section{Introduction}\label{SVRSECSECINTRO}

In this paper, we consider the Severi varieties on a ruled surface over a smooth elliptic curve. 

Let $E$ be a smooth elliptic curve, let $\calE$ be a rank $2$ vector bundle on $E$ given by a nonzero vector in $\Ext(\OO_E, \OO_E)$ and let $R = \PP \calE$. Such surfaces arise naturally when we study the degeneration of abelian and K3 surfaces \cite{ZAHARIUC2022320}. We call such a ruled surface the {\em Atiyah ruled surface} over $E$.

The main purpose of this note is to prove the following: 

\begin{Theorem}\label{K3RATNOTESTHM808}
Let $E$ be a smooth elliptic curve, let $\calE$ be a rank $2$ vector bundle on $E$ given by a nonzero vector in $\Ext(\OO_E, \OO_E)$ and let $R = \PP \calE$. For a line bundle $L$ on $R$, let $V_{R,L,g}\subset |L|$ be the locus of integral curves $C\in |L|$ of geometric genus $g$. Then when $E$ is general, $L$ is ample and $g\ge 1$, for a general member $[C]\in V_{R,L,g}$,
\begin{itemize}
\item if $L.D \ge 2$, $C$ is nodal, and
\item if $L.D = 1$, $C$ has only nodes and/or ordinary triple points as singularities,
\end{itemize}
where $D$ is the unique section of $R$ over $E$ with self intersection $D^2 = 0$.
\end{Theorem}

The Picard group $\Pic(R)$ of $R$ is generated by $D$ and $\pi^* \Pic(E)$, where $\pi: R\to E$ is the projection.

For an integral curve $C\subset R$ of geometric genus $g$ with normalization $f: \widehat{C} \to R$, we know that \cite[Section B, pp. 108-111]{H-M}
\begin{itemize}
\item if
\begin{equation}\label{SVRSECE000}
\deg(c_1(N_f)) - g + 1 = -K_RC + g - 1 > 0,
\end{equation}
a general deformation of $f$ is immersive \cite{arbcorn};
\item if $f$ is immersive and $N_f(-p_1 - p_2) = N_f\otimes \OO_C(-p_1 - p_2)$ is base point free for all $p_1\ne p_2\in C$, then $\varphi(\Gamma)$ is nodal for a general deformation $\varphi: \Gamma\to R$ of $f$. This is guaranteed if
$\deg (c_1(N_f)) \ge 2g + 2$, i.e.,
\begin{equation}\label{SVRSECE001}
-K_RC \ge 4.
\end{equation}
Thus, as long as we have \eqref{SVRSECE001}, $\varphi(\Gamma)$ is nodal for a general deformation $\varphi: \Gamma\to R$ of $f$.
\end{itemize}

Consequently, our main theorem holds for every $L = mD + \pi^* M$ if $m > 0$ and $\deg M \ge 2$. Therefore, the only remaining case for Theorem \ref{K3RATNOTESTHM808} is $m > 0$ and $\deg M = 1$.
Furthermore, we will show that the case $g\ge 2$ can be reduced to $g=1$ by a degeneration argument. That is, it suffices to prove the theorem for $L = mD + R_p$ and $g = 1$, where $R_p = \pi^* p$ is the fiber of $R$ over a point $p\in E$. Indeed, we have a more precise statement for this case:

\begin{Theorem}\label{K3RATNOTESTHM809}
Let $E$ be a smooth elliptic curve, let $\calE$ be a rank $2$ vector bundle on $E$ given by a nonzero vector in $\Ext(\OO_E, \OO_E)$ and let $R = \PP \calE$. When $E$ is general, for $L = mD + R_p$ and every $[C]\in V_{R,L,1}$, 
\begin{itemize}
\item if $4\nmid m$, $C$ is nodal, and
\item if $4\mid m$, $C$ has only nodes and/or ordinary triple points as singularities,
\end{itemize}
where $D$ is the unique section of $R$ over $E$ with self intersection $D^2 = 0$ and $R_p$ is the fiber of $R$ over $p\in E$. In addition, the triple points do appear as the singularities of some $[C]\in V_{R,L,1}$ if $4\mid m$.
\end{Theorem}

Such elliptic curves were also studied by E. Sernesi in \cite{sernesi2023treibichverdier}.

\subsection*{Conventions}

We work exclusively over $\CC$.

\section{Atiyah Ruled Surface $\PP \calE$}

We start with some basic facts about the Atiyah ruled surface $\PP\calE$.

\begin{Proposition}\label{K3RATNOTESPROP802}
Let $E$ be a smooth elliptic curve, let $\calE$ be a rank $2$ vector bundle on $E$ given by a nonzero vector in $\Ext(\OO_E, \OO_E)$, let $R = \PP \calE$ and let $D\subset R$ be the section of $R/E$ with $D^2 = 0$. Then
\begin{enumerate}
\item For every point $p\in E$, $|D + R_p|$ is a pencil such that every curve $C\ne D \cup R_p \in |D+R_p|$ is a smooth elliptic curve and any pair $C_1\ne C_2\in |D+R_p|$ of curves meet only at $p$ with multiplicity $2$, where $R_p$ is the fiber of $R$ over $p\in E$.
\item For every point $p\in E$, $R\backslash (D\cup R_p) \cong (E\backslash \{p\})\times \BA^1$.
\item For every pair of points $p\ne q\in E$, $R\backslash D$ is isomorphic to the gluing of $(E\backslash \{p\})\times \BA^1$ and $(E\backslash \{q\})\times \BA^1$ via an automorphism
$$
\begin{tikzcd}
(E\backslash \{p,q\})\times \BA^1 \ar{r}{\eta} & (E\backslash \{p,q\})\times \BA^1
\end{tikzcd}
$$
given by
$$
\eta(z, s) = (z, s + h(z))
$$ 
where $h(z)$ is a meromorphic function on $E$ with simple poles at $p$ and $q$.
\item There is an exact sequence of group schemes
\begin{equation}\label{K3RATNOTESE822}
\begin{tikzcd}
0 \ar{r} & {\mathbf G}_a \ar{r} & \Aut(R)_0 \ar{r} & \Aut(D)_0 \ar[equal]{d} \ar{r} & 0\\
&&& \Aut(E)_0
\end{tikzcd}
\end{equation}
where ${\mathbf G}_a$ is the additive group of $\CC$, $\Aut(R)_0$ and $\Aut(E)_0$ are the connected components of $\Aut(R)$ and $\Aut(E)$ containing the identity, respectively.  
Every $\phi\in \Aut(R)_0$ is given by
\begin{equation}\label{K3RATNOTESE824}
\begin{aligned}
\phi(z, s) &= (z + \tau, s + b_1(z))\hspace{24pt} \text{on } (E\backslash \{p, p-\tau\})\times \BA^1\\
\phi(z, s) &= (z + \tau, s + b_2(z)) \hspace{24pt} \text{on } (E\backslash \{q, q-\tau\})\times \BA^1
\end{aligned}
\end{equation}
where $\tau\in \Pic^0(E) = J(E)$, $p$ and $q$ are two distinct points on $E$ satisfying $p-q\ne \pm \tau$, $b_1(z)$ is a meromorphic function on $E$ with simple poles at $p$ and $p-\tau$, $b_2(z)$ is a meromorphic function on $F$ with simple poles at $q$ and $q-\tau$, and $b_1(z)$ and $b_2(z)$ satisfy
\begin{equation}\label{K3RATNOTESE823}
b_1(z) + h(z) = b_2(z) + h(z+\tau)
\end{equation}
on $E\backslash \{p, p-\tau, q, q-\tau\}$ with $h(z)$ given in (3).
\end{enumerate}
\end{Proposition}

\begin{proof}
By the exact sequence
$$
\begin{tikzcd}
0 \ar{r} & \OO_E \ar{r} & \calE \ar{r} & \OO_E \ar{r} & 0
\end{tikzcd}
$$
we obtain
$$
h^0(\calE^\vee\otimes \OO_E(p)) = h^0(\OO_E(p)) + h^0(\OO_E(p)) = 2
$$
and hence $|D + R_p|$ is a pencil. Since
$$
\OO_R(D+R_p) \Big|_D = \OO_E(p)
$$
every $C\in |D + R_p|$ passes through $p$. If $C$ is reducible, $C$ must contain a section of $R/E$ and hence it must contain $D$. Consequently, the only reducible member of $|D+R_p|$ is $D\cup R_p$. Every other member of $|D+R_p|$ is a section of $R/E$. For $C_1\ne C_2\in |D+R_p|$, one of $C_1$ and $C_2$ must be integral. Let us assume that $C_1$ is a section of $R/E$. Then
$$
\OO_{C_1}(C_2) = \OO_{C_1}(D+ R_p) = \OO_{C_1}(2p).
$$
We know that both $C_1$ and $C_2$ pass through $p$ and they have intersection number $2$. So $C_1.C_2 = p + p'$. Then $p+p'\sim_{\text{rat}} 2p$ on $C_1$ and hence $p' = p$. That is, $C_1$ and $C_2$ meet at $p$ with multiplicity $2$ and they do not have any other intersections. This proves (1).

Let $\alpha_p: R\dashrightarrow \PP^1$ be the rational map given by the pencil $|D+R_p|$. By (1), the map
$$
\begin{tikzcd}
R\backslash (D\cup R_p) \ar{r}{\pi\times \alpha_p}[below]{\cong} & (E\backslash \{p\}) \times \BA^1
\end{tikzcd}
$$
is an isomorphism, where $\pi: R\to F$ is the projection. This proves (2).

We have
$$
R\backslash D = (R\backslash (D\cup R_p)) \cup (R\backslash (D\cup R_q))
$$
with $(R\backslash (D\cup R_p))$ and $(R\backslash (D\cup R_q))$ isomorphic to $(E\backslash \{p\}) \times \BA^1$ and $(E\backslash \{q\}) \times \BA^1$ via $\pi\times \alpha_p$ and $\pi\times \alpha_q$, respectively. So $R\backslash D$ is the gluing of $(E\backslash \{p\}) \times \BA^1$ and $(E\backslash \{q\}) \times \BA^1$ via an automorphism $\eta \in \Aut(U\times \BA^1/U)$
$$
\begin{tikzcd}
U\times \BA^1 \ar{r}[above]{\eta} & U\times \BA^1
\end{tikzcd}
$$
for $U = E\backslash \{p, q\}$. Such an automorphism is given by
$$
\eta(z, s) = (z, h(z) s + f(z))
$$
where $h(z)$ and $f(z)$ are meromorphic functions on $E$ such that they are holomorphic on $U$ and $h(z)\ne 0$ on $U$. So $h(z)$ has zeros and poles only at $p$ and $q$ and $f(z)$ has poles only at $p$ and $q$.

A member of the pencil $|D + R_p|$ other than $D\cup R_p$ is given by
$$ (\pi\times \alpha_p)^{-1} \{s = a\} $$
for $a\in \CC$. Similarly, a member of the pencil $|D + R_q|$ other than $D\cup R_q$ is given by
$$ (\pi\times \alpha_q)^{-1} \{s = b\} $$
for $b\in \CC$. These two curves meet at two points lying in $R\backslash (D\cup R_p \cup R_q)$. Therefore,
$$
\{s = a\} \cap \eta^{-1} \{s = b\}
$$
has two intersections (counted with multiplicity) in $U\times \BA^1$ for all $a,b\in \CC$. That is, the function
$$
a h(z) + f(z) - b
$$
has exactly two zeros over $U$ for all $a, b$. It follows that $h(z)$ is a nonzero constant and $f(z)$ has simple poles at $p$ and $q$. We may choose $h(z) \equiv 1$. This proves (3).

Clearly, every automorphism of $R$ preserves the section $D$.
Let $\phi: R\to R$ be an automorphism of $R$ in the kernel of $\Aut(R)\to \Aut(D)$ and let $\phi_1$ and $\phi_2$ be the restriction of $\phi$ to $(E\backslash \{p\}) \times \BA^1$ and $(E\backslash \{q\}) \times \BA^1$, respectively. Suppose that
$\phi_1$ and $\phi_2$ are given by
$$
\begin{aligned}
\phi_1(z,s) &= (z, a_1(z) s + b_1(z))\\
\phi_2(z,s) &= (z, a_2(z) s + b_2(z))
\end{aligned}
$$
where $a_1(z)$ and $b_1(z)$ are meromorphic functions on $E$ with poles at $p$, $a_2(z)$ and $b_2(z)$ are meromorphic functions on $E$ with poles at $q$, $a_1(z)\ne 0$ on $E\backslash \{p\}$ and $a_2(z)\ne 0$ on $E\backslash \{q\}$. Clearly, $a_1(z) \equiv a_1$ and $a_2(z) \equiv a_2$ must be constants. In addition, since $\phi_1\circ \eta = \eta\circ \phi_2$, we have
$$
a_1 (s+h(z)) + b_1(z) = a_2 s + b_2(z) + h(z)
$$
on $(E\backslash \{p,q\})\times \BA^1$. Obviously, $a_1=a_2 = a$ and hence
$$ b_1(z) - b_2(z) = (1-a) h(z). $$
Since $h(z)$ has simple poles at $p$ and $q$, $b_1(z)$ has a single pole at $p$ and $b_2(z)$ has a single pole at $q$, $b_1(z)$ and $b_2(z)$ must have simple poles at $p$ and $q$, respectively, and hence they must be constant. It follows that $a=1$ and $b_1(z) \equiv b_2(z) \equiv b$. This proves that
$$
{\mathbf G}_a = \ker(\Aut(R)\to \Aut(D)).
$$
To complete the proof of \eqref{K3RATNOTESE822}, it remains to prove that the map
$$
\begin{tikzcd}
\Aut(R)_0\ar[two heads]{r} & \Aut(D)_0
\end{tikzcd}
$$
is surjective. 

Every automorphism $\lambda\in \Aut(E)_0$ is given by a translation $\lambda(p) = p + \tau$ for some $\tau\in \Pic^0(E) = J(E)$.

For a given $\tau\in J(E)$, if there exist a pair of meromorphic functions $b_1(z)$ and $b_2(z)$ satisfying \eqref{K3RATNOTESE823}, then $\phi\in \Aut(R)_0$ given by \eqref{K3RATNOTESE824} maps to $\lambda\in \Aut(E)_0$ with $\lambda(p) = p+\tau$. So it suffices to prove the existence of $b_1(z)$ and $b_2(z)$ satisfying \eqref{K3RATNOTESE823}.

If $\tau = 0$, we can simply take $b_1(z) \equiv b_2(z) \equiv b$ to be a constant.

Suppose that $\tau \ne 0$. We lift \eqref{K3RATNOTESE823} from $E\cong \CC/\Lambda$ to $\CC$. Then $b_1(z),b_2(z)$ and $h(z)$ are doubly periodic meromorphic functions on $\CC$. We choose $b_1(z)$ such that 
$$
\operatorname{Res}_p b_1(z) = -\operatorname{Res}_p h(z).
$$
Since
$$
\operatorname{Res}_p b_1(z) + \operatorname{Res}_{p-\tau} b_1(z) = 0
$$
we have
$$
\operatorname{Res}_{p-\tau} b_1(z) = \operatorname{Res}_p h(z)
= \operatorname{Res}_{p-\tau} h(z+\tau).
$$
So $b_2(z) = b_1(z) + h(z) - h(z+\tau)$ is analytic at $p$ and $p-\tau$. This proves the existence of $b_1(z)$ and $b_2(z)$ satisfying \eqref{K3RATNOTESE823} and hence (4).
\end{proof}

Let $C\in |mD + R_p|$ be a (possibly singular) elliptic curve on $R$ and let $\nu: \calC \to R$ be the normalization of $C$. We let
$$
S = \calC \times_E R = \PP (\pi\circ \nu)^* \calE
$$
via the maps $\pi\circ \nu: \calC\to E$ and $\pi: R\to E$. Clearly, $(\pi\circ \nu)^* \calE$ is a rank $2$ vector bundle on $\calC$ given by a nonzero vector in $\Ext(\OO_\calC,\OO_\calC)$.

The map $g: S\to R$ is induced by $\pi\circ \nu: \calC\to E$ and is hence \'etale. Let us consider the preimage
$$
g^{-1}(C) = \calC\times_E C
$$
of $C$. It contains the curve $G = \{(s, \nu(s)): s\in \calC\}\cong \calC$. It is not hard to see that $G\in |\OO_S(\calD + S_q)|$, where $\calD = g^* D$ is the unique section of $S/\calC$ with self intersection $0$, $q\in (\pi\circ \nu)^{-1}(p)$ and $S_q$ is the fiber of $S/\calC$ over $q$.

Since $g: S\to R$ is Galois,
$$
g^* C = \sum_{\sigma\in \Aut(S/R)} \sigma(G).
$$
The map $g: g^* C \to C$ is \'etale. So $C$ is nodal if and only if $g^* C$ is, i.e., it has normal crossings. 

Since $h = \pi\circ \nu: \calC \to E$ is an isogeny, the dual isogeny $h^\vee: E\to \calC$ has the property that $h^\vee \circ h: \calC\to \calC$ is a multiplication map given by $x\to p + n(x-p)$ for some integer $n$. So the Galois group $\Aut(\calC/E)$ is a subgroup of $\Aut(h^\vee\circ h)$. Hence $\Aut(\calC/E)$ is given by a finite subgroup of $J(\calC) = \Pic^0(\calC)$. That is, every $\sigma\in \Aut(\calC/E)$ is given by a translation $\sigma(x) = x + \tau$ for some torsion $\tau\in J(\calC)$.

To prove Theorem \ref{K3RATNOTESTHM808}, it suffices to prove the following:

\begin{Proposition}\label{K3RATNOTESPROP803}
Let $E$ be a smooth elliptic curve, let $\calE$ be a rank $2$ vector bundle on $E$ given by a nonzero vector in $\Ext(\OO_E, \OO_E)$, let $R = \PP \calE$, let $D\subset R$ be the section of $R/E$ with $D^2 = 0$ and let $A\subset \Aut(R)_0$ be a finite subgroup of $\Aut(R)_0$ acting freely on $R$. Then when $E$ is general, for every point $p\in E$ and every smooth curve $G\in |D + R_p|$,
$$
\sum_{\sigma\in A} \sigma(G)
$$
has normal crossings if $A$ does not contain the subgroup
$$
J(E)_{2} = \big\{\tau\in J(E): 2\tau = 0\big\} \cong \ZZ/2\ZZ\times \ZZ/2\ZZ
$$
and has only nodes and ordinary triple points as singularities otherwise.
\end{Proposition}

When $C\in |mD + E|$, the Galois group $\Aut(\calC/E)$ has order $m$. If $4\nmid m$, $\Aut(\calC/E)$ does not contain a subgroup of order $4$ and hence $C$ is nodal by the above proposition.

Here we let
$$
\begin{aligned}
J(E)_n &= \big\{\tau\in J(E): n\tau = 0\big\} \cong \ZZ/n\ZZ\times \ZZ/n\ZZ \hspace{12pt}\text{and}\\
J(E)_\text{tors} &= \bigcup_{n=1}^\infty J(E)_n
\end{aligned}
$$
be the torsion subgroups of $J(E)$. For every $\tau\in J(E)_\text{tors}$, we define the order $\ord(\tau)$ of $\tau$ to be the smallest positive integer $n$ such that $n\tau = 0$ and let $\ord(\tau) = \infty$ if $\tau\not\in J(E)_\text{tors}$.

Let $\phi\in \Aut(R)_0$ be an automorphism of order $n$. By \eqref{K3RATNOTESE824}, $\phi$ is given by a meromorphic function $b_1(z)$ on $E$ with simple poles at $p$ and $p-\tau$ satisfying
\begin{equation}\label{SVRSECE002}
b_1(z) + b_1(z+\tau) + ... + b_1(z+(n-1)\tau) = 0
\end{equation}
where $\tau\in J(E)_\text{tors}$ has order $\ord(\tau) = n$.

To prove that $G$ and $\phi(G)$ intersect transversely, it suffices to prove that
$b_1(z)$ does not have a zero of multiplicity $2$, i.e.,
\begin{equation}\label{SVRSECE003}
b_1(p - \eta) \ne 0 \hspace{12pt} \text{for } \tau = 2\eta
\end{equation}
when $E$ is a general elliptic curve.

Let $\phi_1\ne \phi_2\in \Aut(R)_0$ be two automorphisms of finite order. Similarly, $\phi_1$ and $\phi_2$ are given by two meromorphic functions $b_1(z)$ and $b_2(z)$ on $E$ with simple poles at $\{p,p-\tau_1\}$ and $\{p, p-\tau_2\}$, respectively, satisfying
\begin{equation}\label{SVRSECE004}
b_i(z) + b_i(z+\tau_i) + ... + b_i(z+(n_i-1)\tau_i) = 0
\end{equation}
for $i=1,2$, where $\tau_i\in J(E)_\text{tors}$ has order $n_i$ and $\tau_1\ne \tau_2$.
To show that $G, \phi_1(G)$ and $\phi_2(G)$ do not meet at one point, it suffices to show that
\begin{equation}\label{SVRSECE005}
\{b_1(z) = 0\} \cap \{b_2(z) = 0\} = \emptyset
\end{equation}
where $E$ is a general elliptic curve. So it remains to prove \eqref{SVRSECE003} and \eqref{SVRSECE005}. 

Let us start with the observation that the meromorphic functions $b_i(z)$ satisfying \eqref{SVRSECE004} are unique up to a scalar, depending only on $p$ and $\tau_i$.

\begin{Proposition}\label{SVRSECPROP002}
Let $E$ be an elliptic curve and let $p$ be a point of $E$. For every $\tau\in J(E)_{\mathrm{tors}}$ of order $n$ and every meromorphic function $b(z)$ on $E$ with simple poles at $p$ and $p-\tau$ and no other poles,
$$
\sum_{k=0}^{n-1} b(z + k \tau)
$$
is constant.

In addition, there is a unique meromorphic function $b(z) = b_{\tau, p}(z)$ on $E$, up to a scalar, with simple poles at $p$ and $p-\tau$ and no other poles such that
\begin{equation}\label{SVRSECE018}
\sum_{k=0}^{n-1} b(z + k \tau) = 0.
\end{equation}

Furthermore, for all positive integers $m$ with $n\mid m$ and every meromorphic function $b(z)$ on $E$ with simple poles at $p$ and $p-\tau$ and no other poles,
\begin{equation}\label{SVRSECE019}
\sum_{\lambda\in J(E)_m} b(z+\lambda) = \dfrac{m^2}{n} \sum_{k=0}^{n-1} b(z + k \tau).
\end{equation}
Consequently, \eqref{SVRSECE018} holds if and only if
\begin{equation}\label{SVRSECE020}
\sum_{\lambda\in J(E)_m} b(z + \lambda) = 0
\end{equation}
for some positive integer $m$ with $n\mid m$.
\end{Proposition}

\begin{proof}
Let $\omega\in H^0(\Omega_E)$ be a nonzero holomorphic $1$-form on $E$. Then $b(z)\omega$ is a meromorphic $1$-form on $E$ with simple poles at $p$ and $p-\tau$. So
$$
\operatorname{Res}_{p} b(z)\omega + \operatorname{Res}_{p-\tau} b(z)\omega = 0.
$$
It follows that
$$
\sum_{k=0}^{n-1} b(z + k \tau) \omega
$$
is a holomorphic $1$-form on $E$ and hence
$$
\sum_{k=0}^{n-1} b(z + k \tau)
$$
is constant on $E$.

Let $V = H^0(\OO_E(p_1 + p_2))\cong {\mathbb C}^2$ be the vector space of meromorphic functions on $E$ with at worst simple poles at $p_1 = p$ and $p_2 = p-\tau$ and let $L: V\to {\mathbb C}$ be the map given by
$$
L(b(z)) = \sum_{k=0}^{n-1} b(z + k \tau).
$$
Clearly, $L$ is linear. When $b(z) \equiv c$ is constant, $L(b(z)) = nc$ and hence $L$ is surjective. Thus, $\ker(L)$ is a one-dimensional subspace of $V$. So there exists a unique $b(z)\in V$, up to a scalar, such that
$$
\sum_{k=0}^{n-1} b(z + k \tau) = 0.
$$

Obviously, $G = \{k\tau: k\in \ZZ\}$ is a subgroup of $J(E)_m$ for $n\mid m$. So
$$
J(E)_m = \bigsqcup_{i=1}^d (\lambda_i + G)
$$
for some $\lambda_1,\lambda_2,...,\lambda_d\in J(E)_m$ and $d = m^2/n$. Then
$$
\sum_{\lambda\in J(E)_m} b(z + \lambda) = \sum_{i=1}^d \sum_{\lambda\in G} b(z+\lambda_i + \lambda)
$$
We have proved that
$$
\sum_{\lambda\in G} b(z + \lambda)
$$
is constant. Therefore,
$$
\sum_{\lambda\in G} b(z + \lambda) \equiv \sum_{\lambda\in G} b(z+\lambda_i + \lambda)
$$
for all $i$ and hence
$$
\sum_{\lambda\in J(E)_m} b(z + \lambda) = \sum_{i=1}^d \sum_{\lambda\in G} b(z+\lambda_i + \lambda)
= d\sum_{\lambda\in G} b(z + \lambda).
$$
This proves \eqref{SVRSECE019}.
\end{proof}

Thus, \eqref{SVRSECE003} becomes

\begin{Proposition}\label{SVRSECPROP000}
For a general elliptic curve $E$, every point $p\in E$, every $\tau\in J(E)_{\mathrm{tors}}$ of order $n \ge 2$
and every $\eta \in J(E)_{\mathrm{tors}}$ satisfying $2\eta = \tau$, we have
$$
b_{\tau,p}(p - \eta) \ne 0
$$
where $b_{\tau,p}(z)$ is the meromorphic function on $E$ given in Proposition \ref{SVRSECPROP002}.
\end{Proposition}

Similarly, a more precise statement of \eqref{SVRSECE005} is

\begin{Proposition}\label{SVRSECPROP001}
Let $E$ be an elliptic curve, let $p\in E$ be a point on $E$ and let
$b_{\tau,p}$ be the meromorphic function on $E$ given in Proposition \ref{SVRSECPROP002} for a nonzero torsion $\tau\in J(E)_{\mathrm{tors}}$.

For $E$ general and any two torsions $\tau_1\ne \tau_2\in J(E)_{\mathrm{tors}}$ of orders $n_1\ge 2$ and $n_2 \ge 2$, respectively, one of the following holds:
\begin{equation}\label{SVRSECE024}
\{b_{\tau_1,p}(z) = 0\} \cap \{b_{\tau_2,p}(z) = 0\} = \emptyset
\end{equation}
or
\begin{equation}\label{SVRSECE028}
(n_1,n_2) = (2,2)
\end{equation}
or
\begin{equation}\label{SVRSECE029}
(n_1,n_2) = (6,6),\ \langle \tau_1, \tau_2\rangle = 3 \text{ in } J(E)_6 \text{ and } \ord(\tau_1 - \tau_2) = 6.
\end{equation}
In addition, when $(n_1,n_2) = (2,2)$,
\begin{equation}\label{SVRSECE025}
\{b_{\tau_1,p}(z) = 0\} \cap \{b_{\tau_2,p}(z) = 0\} = \{ p - \tau_3 \}
\end{equation}
where $\tau_3\in J(E)_{\mathrm{tors}}$ is a torsion of order $2$ different from $\tau_1$ and $\tau_2$.

For $E$ general and any three distinct nonzero torsions $\tau_1, \tau_2, \tau_3\in J(E)_{\mathrm{tors}}$,
\begin{equation}\label{SVRSECE026}
\{b_{\tau_1,p}(z) = 0\} \cap \{b_{\tau_2,p}(z) = 0\} \cap \{b_{\tau_3,p}(z) = 0\} = \emptyset.
\end{equation}
\end{Proposition}

The intersection pairing $\langle \bullet,\bullet\rangle$ on $J(E)_n$ will be defined in the next section.

Let us explain how Propositions \ref{SVRSECPROP000} and \ref{SVRSECPROP001} imply Proposition \ref{K3RATNOTESPROP803}. Proposition \ref{SVRSECPROP000} implies that any pair curves among
$\{\sigma(G): \sigma\in A\}$ meet transversely and thus $\sum \sigma(G)$ has only ordinary singularities, i.e., singularities whose local branches are smooth and meet transversely pairwise. Then Proposition \ref{SVRSECPROP000} says that no three curves among $\{\sigma(G): \sigma\in A\}$ meet at one point with the exceptions \eqref{SVRSECE028} and \eqref{SVRSECE029}, in which cases no more than three curves among $\{\sigma(G): \sigma\in A\}$ meet at one point by \eqref{SVRSECE026}. 
In case \eqref{SVRSECE028}, $\tau_1$ and $\tau_2$ generate $J(E)_2\subset A$.
In case \eqref{SVRSECE029}, $\tau_1$ and $\tau_2$ generate a subgroup of $J(E)_6$ of order $12$ contained in $A$; such a subgroup clearly contains $J(E)_2$.

\section{Torsions on Generic Elliptic Curves}\label{SVRSECSECTOSION}

We will prove Proposition \ref{SVRSECPROP000} and \ref{SVRSECPROP001} by letting $E$ vary in a complete family of elliptic curves $X/B$ with a unique section $P$. There are many choices of such $X$. Let us choose $X$ to be a K3 surface with Picard lattice
$$
\begin{bmatrix}
-2 & 1\\
1 & 0
\end{bmatrix}
$$
We call such $X$ a {\em Bryan-Leung K3} \cite{BryanLeung}. Such $X$ admits an elliptic fiberation $\pi: X\to B = \PP^1$. For $X$ general, it has $24$ nodal fibers over $S \subset B$. The $(-2)$-curve $P\subset X$ is the only section of $\pi$. For each positive integer $n$, let us consider
\begin{equation}\label{SVRSECE006}
\Sigma_n = \overline{\big\{ q\in X_b: b\not\in S,\ \ord(p-q) = n \text{ for } p = P_b = P\cap X_b\big\}}
\end{equation}
Clearly, $\Sigma_n$ is a multi-section of $X/B$ of degree
$$
n^2 \prod_{\substack{p \text{ prime}\\ p\mid n}}\left(1 - \dfrac{1}{p^2}\right)
$$
We claim that $\Sigma_n$ is irreducible. This is proved by studying the monodromy action of $\pi_1(B\backslash S)$ on $\Sigma_n$. Actually, the monodromy action of $\pi_1(B\backslash S)$ on $\Sigma_n$ is induced by its action on $H^1(X_b, \ZZ)$.

Fix a smooth fiber $E = X_b$ of $X$ over $b\in B^\circ = B\backslash S$ and let us consider the monodromy action of $\pi_1(B^\circ)$ on $J(E)_{\text{tors}}$ and $H^1(E, \ZZ)$. From the exponential sequence, we have the diagram
$$
\begin{tikzcd}
& & & J(E)_{n}\ar[hook]{d}\\
0 \ar{r} & H^1(E,\ZZ)\ar[hook]{d}{\times n} \ar{r} & H^1(\OO_E) \ar[equal]{d}{\times n}\ar{r} & J(E) \ar[two heads]{d}{\times n}\ar{r} & 0\\
0 \ar{r} & H^1(E,\ZZ) \ar[two heads]{d}\ar{r} & H^1(\OO_E) \ar{r} & J(E) \ar{r} & 0\\
& H^1(E,\ZZ)/nH^1(E,\ZZ) 
\end{tikzcd}
$$
Thus, we have
$$
J(E)_{n} \cong H^1(E,\ZZ)/nH^1(E,\ZZ)
$$
and the action of $\pi_1(B^\circ)$ on $J(E)_\text{tors}$ is induced by its action of $H^1(E,\ZZ)$.

The action $\pi_1(B^\circ)$ on $H^1(E,\ZZ)$ preserves the intersection product of $H^1(E,\ZZ)$. Thus, it is given by a group homomorphism
$$
\begin{tikzcd}
\pi_1(B^\circ) \ar{r} & \Aut(H^1(E,\ZZ)) \cong \operatorname{SL}_2(\ZZ)
\end{tikzcd}
$$
where $\Aut(H^1(E,\ZZ))$ is the automorphism of $H^1(E,\ZZ)$ as a lattice. Thus, the induced action of $\pi_1(B^\circ)$ on $\Sigma_n$ is given by the group homomorphism
$$
\begin{tikzcd}
\pi_1(B^\circ) \ar{r} \ar{dr} & \ar[two heads]{d} \operatorname{SL}_2(\ZZ)\\
& \operatorname{SL}_2(\ZZ/n\ZZ)
\end{tikzcd}
$$

\begin{Proposition}\label{SVRSECPROP003}
Let $\pi: X\to B = \PP^1$ be a Bryan-Leung K3 surface with $24$ nodal fibers. Then the monodromy action
$\pi_1(B^\circ)\to \operatorname{SL}_2(\ZZ/n\ZZ)$ is surjective and $\Sigma_n$ is irreducible for all $n\in \ZZ^+$ with $\Sigma_n \subset X$ defined by \eqref{SVRSECE006}.
\end{Proposition}

The action of $\pi_1(B^\circ)$ on $H^1(E,\ZZ)$ is well understood. At each $b_i\in \{b_1,b_2,...,b_{24}\}$, the loop around $b_i$ acts on $H^1(E,\ZZ)$ by a Lefschetz-Picard transform (cf. \cite{SurveyHodgeConjecture}):
$$
T_{\delta_i} (\lambda) = \lambda + \langle \lambda,\delta_i\rangle \delta_i
$$
where $\delta_i\in H^1(E,\ZZ)$ is called the {\em vanishing cycle} at the node of $X_{b_i}$ for $i=1,2,...,24$ and $\langle \bullet, \bullet\rangle$ is the intersection pairing on $H^1(E,\ZZ)$. The monodromy action of $\pi_1(B^\circ)$ on $H^1(E, \ZZ)$ is the subgroup of $\Aut(H^1(E, \ZZ))$ generated by $T_{\delta_1}, T_{\delta_2}, ..., T_{\delta_{24}}$. Clearly, $T_{\delta_i}$ lift to actions on $H^1(E,\ZZ)/nH^1(E,\ZZ)$. We start with a simple observation:

\begin{Lemma}\label{SVRSECLEM000}
Let $\delta_1, \delta_2, ..., \delta_{24}\in H^1(X_b, \ZZ)$ be the vanishing cycles associated to a Bryan-Leung K3 surface $\pi: X\to B=\PP^1$ with $24$ nodal fibers. Then
\begin{enumerate}
\item $\delta_i$ are indivisible, i.e., there do not exist $\eta\in H^1(X_b, \ZZ)$ and an integer $m\ge 2$ such that $\delta_i = m\eta$;
\item for every indivisible $\lambda\in H^1(X_b, \ZZ)$,
$$
\gcd(\langle \lambda, \delta_1\rangle, \langle \lambda, \delta_2\rangle, ..., \langle \lambda, \delta_{24}\rangle) = 1.
$$
\end{enumerate}
\end{Lemma}

\begin{proof}
It is well known that $\delta_i$ are indivisible (cf. \cite[Example 6.6, p. 72]{SurveyHodgeConjecture}), as a consequence of the smoothness of $X$. Here we gives another argument based on torsion points.

Suppose that $\delta/m\in H^1(E, \ZZ)$ for some $\delta = \delta_i$ and $m\ge 2$. For simplicity, let us assume that $m$ is prime. Then $H^1(E,\ZZ)/mH^1(E,\ZZ)$ is fixed by $T_{\delta}$ so $\Sigma_m$ is the union $Q_1\cup Q_2\cup ...\cup Q_{m^2-1}$ of $m^2 - 1$ local sections over a disk $U\subset B$ around the point $s = b_i\in S$. Since $X$ is smooth, each $Q_j$ meets $X_{s}$ at a point away from the node $x$ of $X_{s}$. Let $f: X\dashrightarrow X$ be the rational map given by $f(q) = p+m(q-p)$ for $q\in X_b$, $b\in B^\circ$ and $p = P\cap X_b$. Then $f$ can be extended to a regular, quasi-finite and unramified morphism
$$
\begin{tikzcd}
X\backslash \{x_1,x_2,...,x_{24}\} \ar{r}{f} & X
\end{tikzcd}
$$
where $x_1,x_2,...,x_{24}$ are the nodes of the $24$ fibers $X_S = \pi^{-1}(S)$. Then
$$
X_U \cap f^{-1}(P) = P\cup Q_1\cup Q_2\cup ... \cup Q_{m^2-1}
$$
for $X_U = \pi^{-1}(U)$. Since $f$ is unramified, $P, Q_1, Q_2,...,Q_{m^2-1}$ are disjoint. Therefore, $p = P\cap X_s$ and $q_j = Q_j\cap X_s$ are $m^2$ distinct points on $X_s\backslash \{x\}$. But there are only $m$ distinct points $q$ on $X_s\backslash \{x\}$ such that $m(q-p) = 0$ in $\Pic^0(X_s) \cong \CC^*$, which is a contradiction.

For (2), if
$$
\gcd(\langle \lambda, \delta_1\rangle, \langle \lambda, \delta_2\rangle, ..., \langle \lambda, \delta_{24}\rangle) = m \ge 2,
$$
then $\lambda\in H^1(E,\ZZ)/m H^1(E,\ZZ)$ is fixed by $T_{\delta_i}$ for all $i$. Therefore, $\Sigma_m$ contains a section. But $P$ is the only section of $X/B$, which is a contradiction.
\end{proof}

\begin{proof}[Proof of Proposition \ref{SVRSECPROP003}]
If $n = n_1n_2$ for two coprime integers $n_1$ and $n_2$, then the surjectivity of $\pi_1(B^\circ)\to \operatorname{SL}_2(\ZZ/n\ZZ)$ follows from those of $\pi_1(B^\circ)\to \operatorname{SL}_2(\ZZ/n_i\ZZ)$ for $i=1,2$ via the group isomorphism
$$
\begin{tikzcd}
\operatorname{SL}_2(\ZZ/n\ZZ) \ar{r}{\sim} & \operatorname{SL}_2(\ZZ/n_1\ZZ)\times \operatorname{SL}_2(\ZZ/n_2\ZZ)
\end{tikzcd}
$$
So by induction on the number of prime divisors of $n$, it suffices to prove the proposition for $n = p^d$ with $p$ prime.

For simplicity, suppose that $\delta_1 = e_1$, where $\{e_1,e_2\}$ is the standard basis of $\ZZ/n\ZZ \times \ZZ/n\ZZ$. By Lemma \ref{SVRSECLEM000},
$$
\gcd(\langle \delta_1, \delta_2\rangle, \langle \delta_1, \delta_3\rangle, ..., \langle \delta_1, \delta_{24}\rangle) = 1
$$
So there exists $2\le i\le 24$ such that $p\nmid \langle \delta_1, \delta_i\rangle$. We may assume that $p\nmid \langle \delta_1, \delta_2\rangle$. Then $\delta_2 = a e_1 + be_2$ for some $p\nmid b$. Let $m$ be an integer such that $bm \equiv 1\ (\text{mod } n)$. By changing the basis from $\{e_1,e_2\}$ to $\{e_1, ame_1 + e_2\}$, we may assume that $\delta_2 = be_2$.

Clearly,
$$
T_{\delta_1}^{-1}\begin{bmatrix}
x\\
y
\end{bmatrix} = \begin{bmatrix}
1 & 1\\
& 1
\end{bmatrix}\begin{bmatrix}
x\\
y
\end{bmatrix} \hspace{8pt}\text{and}\hspace{8pt}
(T_{\delta_2})^m\begin{bmatrix}
x\\
y
\end{bmatrix} = \begin{bmatrix}
1 & \\
1 & 1
\end{bmatrix}\begin{bmatrix}
x\\
y
\end{bmatrix}\text{ for } \begin{bmatrix}
x\\
y
\end{bmatrix} = xe_1 + ye_2
$$
and hence $\pi_1(B^\circ)\to \operatorname{SL}_2(\ZZ/n\ZZ)$ is surjective.
\end{proof}

Let us consider the degeneration of the function $b_{\tau,p}(z)$ when $X_t$ degenerates to $X_0$ for some $0\in S$.

\begin{Proposition}\label{SVRSECPROP004}
Let $\pi: X\to \Delta$ be a flat projective family of curves over the unit disk $\Delta$
such that $X$ is smooth, $X_t$ is a smooth elliptic curve for $t\ne 0$ and $X_0$ is a rational curve with a node, where $X_t$ is the fiber of $X$ over $t\in \Delta$. Let $P$ and $Q$ be two sections of $X/\Delta$ such that $P_t - Q_t$ is a torsion class in $J(X_t)$ of order $n \ge 2$ for $t\ne 0$. Then there exists an integral curve $Z\subset X$ flat of degree $2$ over $\Delta$ such that $Z_0$ is supported on the node of $X_0$ and
\begin{equation}\label{SVRSECE009}
\big\{b_{\tau,p}(z) = 0\big\} = Z_t
\end{equation}
for $t\ne 0$, where $b_{\tau,p}(z)$ is the meromorphic function on $X_t$ 
given in Proposition \ref{SVRSECPROP002} with $\tau = P_t - Q_t$ and $p = P_t$.
\end{Proposition}

\begin{proof}
Since $P$ and $Q$ are sections of $X/\Delta$ and $X$ is smooth, $P$ and $Q$ meet $X_0$ at smooth points $P_0$ and $Q_0$ of $X_0$. By the argument in the proof of Lemma \ref{SVRSECLEM000}, $P_0-Q_0$ is a torsion class in $\Pic^0(X_0)\cong \CC^*$ of order $n$.

Let us consider $\pi_* \OO_X(P + Q)$. This is a rank $2$ vector bundle over $\Delta$ since $h^0(\OO_{X_t}(P+Q)) = 2$ for all $t$. Therefore, 
$$H^0(\pi_* \OO_X(P+Q)) = H^0(\OO_X(P+Q))$$
is a rank $2$ free module over $\CC[[t]]$.

Let $o$ be the node of $X_0$. Then $X_0\backslash \{o\}\cong \CC^*$. We may assume that $P_0 = 1$ and $Q_0 = \eta = \exp(2\pi i/n)$. Then $H^0(\OO_{X_0}(P_0 + Q_0))$ is spanned by the constant function $1$ and
$$
s_0(z) = \dfrac{z}{(z-1)(z-\eta)}
$$
over $\CC$. We can choose $s\in H^0(\OO_X(P+Q))$ such that $s_0$ is the restriction of $s$ to $X_0$, i.e., $s_0(z) = s(0, z)$, where we consider $s = s(t, z)$ as a meromorphic function on $X$ with simple poles along $P$ and $Q$. Then $H^0(\OO_X(P+Q))$ is generated by $1$ and $s$ over $\CC[[t]]$.

Let $\phi: X\backslash\{o\}\to X\backslash\{o\}$ be the automorphism given by $\phi(z) = z + (p-q)$ for $z\in X_t$, $p=P_t$ and $q = Q_t$. Then
$$
\sum_{k=0}^{n-1} s(t, \phi^k(z))
$$
is constant for each fixed $t\ne 0$ by Proposition \ref{SVRSECPROP002}. For $t=0$, we have
$$
\sum_{k=0}^{n-1} s(0, \phi^k(z)) = \sum_{k=0}^{n-1} \dfrac{\eta^k z}{(\eta^k z - 1)(\eta^k z - \eta)} = 0.
$$
Therefore,
$$
f(t) = \sum_{k=0}^{n-1} s(t, \phi^k(z))
$$
for some $f(t)\in \CC[[t]]$ with $f(0) = 0$. Then $ns(t,z) - f(t)$ is a section of $\OO_X(P+Q)$ whose restriction to $X_t$ is exactly the function $b_{\tau,p}(z)$.

Let
\begin{equation}\label{SVRSECE007}
Z = \Big\{ns(t,z) - f(t) = 0\Big\}
\end{equation}
be the vanishing locus of $ns(t,z) - f(t)$. Then \eqref{SVRSECE009} follows from our choice of $f(t)$. In addition, since $ns(0,z) - f(0) = ns_0(z)$ and $s_0$ only vanishes at the node $o$ of $X_0$, we see that $Z_0$ is supported at $o$.

We know that $Z$ is a closed subscheme of $X$ of pure dimension one and flat of degree $2$ over $\Delta$. So it must be one of the following:
\begin{itemize}
\item $Z$ is supported on a section of $X/\Delta$ with multiplicity $2$;
\item $Z$ is a union of two distinct sections of $X/\Delta$;
\item $Z$ is an irreducible multi-section of degree $2$ over $\Delta$.
\end{itemize}

Since $Z_0$ is supported on the node $o$ of $X_0$ and $X$ is smooth, $Z$ cannot contain any section of $X/\Delta$. Thus, $Z$ must be an integral curve flat of degree $2$ over $\Delta$.
\end{proof}

Proposition \ref{SVRSECPROP000} follows immediately from the above proposition.

\begin{proof}[Proof of Proposition \ref{SVRSECPROP000}]
Suppose that $b_{\tau,p}(p - \eta) = 0$ on a general elliptic curve $E$ for some torsion class $\tau\in J(E)$ of order $n\ge 2$ and $2\eta = \tau$. Then by Proposition \ref{SVRSECPROP003}, this holds for every torsion class $\tau$ of order $n$.

Let $\pi: X\to B = \PP^1$ be a Bryan-Leung K3 surface with $24$ nodal fibers over $S\subset B$. We choose a point $s\in S$ and let $U\subset B$ be an open disk about $s$. Then there exists a section $Q$ of $X_U = \pi^{-1}(U)$ over $U$ such that $P_t - Q_t$ is a torsion class of order $n$ for all $t\in U$. It follows from Proposition \ref{SVRSECPROP004} that $b_{\tau, p}(z)$ has two distinct zeros on $X_t$ for $\tau = P_t - Q_t$ and $p = P_t$, which is a contradiction.
\end{proof}

\section{Proof of Proposition \ref{SVRSECPROP001}}

In this section, we will prove Proposition \ref{SVRSECPROP001}. Combined with Proposition \ref{SVRSECPROP000}, we obtain Proposition \ref{K3RATNOTESPROP803}. Then Theorem \ref{K3RATNOTESTHM809} follows.

We will prove the following two statements in sequence:

\begin{Proposition}\label{SVRSECPROP005}
For a general elliptic curve $E$, a point $p\in E$ and a pair $\tau_1\ne \tau_2\in J(E)_{\mathrm{tors}}$ of torsions of orders $n_1\ge 2$ and $n_2 \ge 2$, respectively, if
$$
\{b_{\tau_1,p}(z) = 0\} \cap \{b_{\tau_2,p}(z) = 0\} \ne \emptyset
$$
then either
\begin{equation}\label{SVRSECE008}
\{p-q: b_{\tau_i,p}(q) = 0\} \subset J(E)_{\mathrm{tors}}
\end{equation}
for $i=1,2$ or
\begin{equation}\label{SVRSECE021}
n_1 = n_2 = 6,\ \langle \tau_1,\tau_2\rangle = 3 \text{ in } J(E)_6 \text{ and } \ord(\tau_1 - \tau_2) = 6.
\end{equation}
\end{Proposition}

\begin{Proposition}\label{SVRSECPROP006}
For a general elliptic curve $E$, a point $p\in E$ and a pair $\tau_1\ne \tau_2\in J(E)_{\mathrm{tors}}$ of nonzero torsions, if
\begin{equation}\label{SVRSECE010}
\begin{aligned}
\{b_{\tau_1,p}(z) = 0\} \cap \{b_{\tau_2,p}(z) = 0\} &\ne \emptyset \hspace{12pt}\text{and}\\
\{p-q: b_{\tau_i,p}(q) = 0\} &\subset J(E)_{\mathrm{tors}}\hspace{12pt}\text{for } i=1,2,
\end{aligned}
\end{equation}
then $\ord(\tau_1) = \ord(\tau_2) = 2$.
\end{Proposition}

Our main tool is the monodromy action of $\pi_1(B^\circ)$ on $J(E)_{\mathrm{tors}}$. We fix a Bryan-Leung K3 surface $X\to B = \PP^1$ with $24$ nodal fibers over $S\subset B$ and a general fiber $E= X_t$ of $X/B$. We extend the monodromy action on $J(E)_{\mathrm{tors}}$ to the triples $(\tau, q_1,q_2)$ with $\tau\in J(E)_{\mathrm{tors}}$ and $\{b_{\tau,p}(z) = 0\} = \{q_1,q_2\}$.

Let us consider the curve
\begin{equation}\label{SVRSECE011}
\begin{aligned}
&\big\{(\tau,q_1,q_2): \tau\in J(X_t)_n,\ t\in B\backslash S,\ q_1,q_2\in X_t,\ \text{and }\\
&\hspace{24pt} \{b_{\tau,p}(z) = 0\} = \{q_1,q_2\} \text{ for } p = P_t \big\} \subset \Pic^0(X/B) \times_B X\times_B X
\end{aligned}
\end{equation}
By Proposition \ref{SVRSECPROP000}, for each fixed $n\ge 2$, there exists a finite set $S_n\subset B$ such that for every $t\not\in S\cup S_n$, $b_{\tau,p}(z)$ has no double zeros on $X_t$. So the curve defined by \eqref{SVRSECE011} is unramified over $B\backslash (S\cup S_n)$ and we have a well-defined monodromy action of $\pi_1(B\backslash (S\cup S_n))$ on such triples $(\tau,q_1,q_2)$ on a general fiber $E = X_t$. Let us use the notation $\lambda(\tau)$ and $\lambda(\tau, q_1,q_2)$ to denote the action of $\lambda\in \pi_1(B\backslash (S\cup S_n))$ on $\tau\in J(E)_\text{tors}$ and $(\tau, q_1,q_2)$.

We start with a few observations.

\begin{Lemma}\label{SVRSECLEM001}
Let $X\to B =\PP^1$ be a Bryan-Leung K3 surface with $24$ nodal fibers and let $E = X_t$ be a general fiber of $X/B$.
Let $\tau\in J(E)_{\mathrm{tors}}$ be a torsion of order $n\ge 2$ and let $q_1,q_2\in E$ be two points given by
$$
\big\{b_{d\tau,p}(z) = 0\big\} = \big\{q_1,q_2\big\}
$$
for some integer $d$ with $d\tau \ne 0$. If $\lambda\in \pi_1(B\backslash (S\cup S_n))$ acts on $J(E)_n$ by
$$
\lambda(\eta) = \eta + \langle \eta, \tau\rangle \tau
$$
for all $\eta\in J(E)_n$, then
$$
\lambda(d\tau, q_1, q_2) = (d\tau, q_2,q_1).
$$
\end{Lemma}

\begin{proof}
Fix a point $0\in S$ and let $\delta$ be the vanishing cycle associated to the nodal fiber $X_0$. If $\tau = \delta$ in $J(E)_n$, then we must have $\lambda = T_\delta$ in $\operatorname{SL}_2(\ZZ/n\ZZ)$, where $T_\delta$ is the Picard-Lefschetz transform associated to $\delta$. Since
$$T_\delta(d\tau) = d\tau,$$
there is a local section $Q\subset X_U = X\times_B U$ over a simply connect open neighborhood $U$ of $0$ such that $P_t - Q_t = d\tau$. Then the lemma follows from Proposition \ref{SVRSECPROP004}.

More generally, by Proposition \ref{SVRSECPROP003}, there exists $\alpha\in \pi_1(B\backslash (S\cup S_n))$ such that $\alpha(\delta) = \tau$. Then $T_\delta = \alpha^{-1}\circ \lambda\circ \alpha$ since 
$$
\begin{aligned}
\alpha^{-1} \circ \lambda \circ \alpha(\eta) &= \alpha^{-1}\big(\alpha(\eta) + \langle \alpha(\eta), \alpha(\delta)\rangle \alpha(\delta)\big)\\
&= \alpha^{-1}\big(\alpha(\eta) + \langle \eta, \delta\rangle \alpha(\delta)\big)\\
&= \alpha^{-1}\circ \alpha (\eta + \langle \eta, \delta\rangle \delta) = T_\delta(\eta).
\end{aligned}
$$
Thus, the lemma follows.
\end{proof}

\begin{Lemma}\label{SVRSECLEM002}
Let $X\to B =\PP^1$ be a Bryan-Leung K3 surface with $24$ nodal fibers and let $E = X_t$ be a general fiber of $X/B$.
Let $\tau_1$ and $\tau_2\in J(E)_\text{tors}$ be two torsions of the same order $n\ge 2$ with
$m = \langle \tau_1,\tau_2\rangle$ in $J(E)_n$, let $n_1,n_2$ be two integers such that
$n\nmid n_i$ and let
$$
\big\{b_{n_1 \tau_1,p}(z) = 0\big\} = \{q_1, q_2\}.
$$
If $b_{n_2\tau_2, p}(q_1) = 0$, then
\begin{equation}\label{SVRSECE012}
\begin{aligned}
b_{n_2(\tau_2 + km \tau_1), p}(q_1) &= 0 \hspace{24pt}\text{for } 2\mid k \text{ and }\\
b_{n_2(\tau_2 + km \tau_1), p}(q_2) &= 0 \hspace{24pt}\text{for } 2\nmid k
\end{aligned}
\end{equation}
If, in addition, $(2\gcd(mn_2,n)) \nmid n$, then $n_1\tau_1 = n_2 \tau_2$.
\end{Lemma}

\begin{proof}
By Proposition \ref{SVRSECPROP003}, we can find $\lambda\in \pi_1(B\backslash (S\cup S_n))$ such that
$$
\lambda(\alpha) = \alpha + \langle \alpha, \tau_1 \rangle \tau_1
$$
for all $\alpha\in J(E)_n$. Then $\lambda(\tau_1) = \tau_1$. Hence
\begin{equation}\label{SVRSECE013}
\begin{aligned}
\lambda^k (n_1\tau_1, q_1, q_2) &= (n_1\tau_1, q_1,q_2) \hspace{24pt}\text{for } 2\mid k\text{ and }\\
\lambda^k (n_1\tau_1, q_1, q_2) &= (n_1\tau_1, q_2,q_1) \hspace{24pt}\text{for } 2\nmid k
\end{aligned}
\end{equation}
by Lemma \ref{SVRSECLEM001}. Obviously,
\begin{equation}\label{SVRSECE014}
\lambda^k (\tau_2) = \tau_2 - km \tau_1
\end{equation}
for all integers $k$. Combining \eqref{SVRSECE013} and \eqref{SVRSECE014}, we obtain \eqref{SVRSECE012}.

If $(2\gcd(mn_2,n)) \nmid n$, then $k_0 = n/\gcd(mn_2,n)$ is odd. Setting $k=k_0$ in \eqref{SVRSECE012}, we obtain
$$
b_{n_2\tau_2,p}(q_2) = b_{n_2(\tau_2 + k_0m \tau_1), p}(q_2) = 0.
$$
On the other hand, we assume that $b_{n_2\tau_2, p}(q_1) = 0$. So
$$
\big\{b_{n_i \tau_i,p}(z) = 0\big\} = \{q_{1}, q_{2}\}
$$
for $i=1,2$. This implies
$$
n_1\tau_1 = (p - q_{1})+(p-q_{2}) = n_2\tau_2. 
$$
\end{proof}

\begin{Lemma}\label{SVRSECLEM003}
Let $E$ be an elliptic curve, let $p$ be a point on $E$ and let $\tau\in J(E)_{\mathrm{tors}}$ be a torsion of order $2$. Then
$$
\big\{b_{\tau,p}(z) = 0\big\} = \big\{q_1, q_2\big\}
$$
such that $\tau$, $p-q_1$ and $p-q_2$ are the three distinct $2$-torsions. 
\end{Lemma}

\begin{proof}
Let $\tau$, $\tau_1$ and $\tau_2$ be the three distinct $2$-torsions. Clearly,
$$
\tau = \tau_1 + \tau_2.
$$
So there exist a rational function $b(z)$ on $E$ with simple poles at $p$ and $p-\tau$ and simple zeros at $p-\tau_1$ and $p-\tau_2$. Note that $b(z+\tau)$ also has simple poles at $p$ and $p-\tau$ and simple zeros at $p-\tau_1$ and $p-\tau_2$. Therefore,
$$
b(z +\tau) \equiv c b(z)
$$
for a constant $c$. And since $b(z) + b(z+\tau)$ is a constant by Proposition \ref{SVRSECPROP002}, we must have $c = -1$ and
$$
b(z) + b(z+\tau) \equiv 0.
$$
Therefore, $b_{\tau,p}(z) \equiv \lambda b(z)$ for a constant $\lambda\ne 0$ by the uniqueness of $b_{\tau,p}(z)$ and the lemma follows.
\end{proof}

\begin{Lemma}\label{SVRSECLEM004}
Let $E$ be an elliptic curve, let $p$ be a point on $E$ and let $\tau_1\ne \tau_2\in J(E)_{\mathrm{tors}}$ be two distinct nonzero torsions. If
$$
\begin{aligned}
\big\{b_{\tau_1,p}(z) = 0\big\} &= \big\{q_1, q_2\big\}\\
\big\{b_{\tau_2,p}(z) = 0\big\} &= \big\{q_1, q_3\big\}
\end{aligned}
$$
then
$$
b_{\tau_1 -\tau_2, p}(q_2) = 0.
$$
\end{Lemma}

\begin{proof}
Note that
$$
q_2 = q_3 - (\tau_1 - \tau_2)
$$
since
$$
\begin{aligned}
\tau_1 &= (p - q_1) + (p - q_2) \hspace{24pt}\text{and}
\\
\tau_2 &= (p - q_1) + (p - q_3).
\end{aligned}
$$
Let us consider the meromorphic function $b_{\tau_2,p}(z+(\tau_1 - \tau_2))$. It has simple poles at
$p - (\tau_1 - \tau_2)$ and $(p-\tau_2) - (\tau_1 - \tau_2) = p - \tau_1$ and a zero at
$$
q_3 - (\tau_1 - \tau_2) = q_2.
$$
Therefore, 
$$ b(z) = b_{\tau_1,p}(z) + c b_{\tau_2,p}(z+(\tau_1 - \tau_2))$$
has simple poles at $p$ and $p-(\tau_1 - \tau_2)$ and a zero at $q_2$ for the constant $c$ given by
$$
c = -\dfrac{\operatorname{Res}_{p-\tau_1} b_{\tau_1,p}(z) \omega}{\operatorname{Res}_{p-\tau_1} b_{\tau_2,p}(z + (\tau_1-\tau_2))\omega}.
$$
where $\omega$ is a nonvanishing holomorphic $1$-form on $E$.

Let $n$ be a positive integer such that $\tau_1,\tau_2\in J(E)_n$. Then
$$
\begin{aligned}
\sum_{\lambda\in J(E)_n} b(z+\lambda) &= \sum_{\lambda\in J(E)_n} b_{\tau_1,p}(z+ \lambda)
+ c \sum_{\lambda\in J(E)_n} b_{\tau_2,p}(z+(\tau_1 - \tau_2) + \lambda)\\
&= \sum_{\lambda\in J(E)_n} b_{\tau_1,p}(z+ \lambda)
+ c \sum_{\lambda\in J(E)_n} b_{\tau_2,p}(z + \lambda) \equiv 0
\end{aligned}
$$
by Proposition \ref{SVRSECPROP002}. Then by the uniqueness of $b_{\tau_1 -\tau_2, p}(z)$, we must have
$b_{\tau_1 -\tau_2, p}(z) \equiv a b(z)$ for some constant $a\ne 0$ and the lemma follows.
\end{proof}

\begin{Lemma}\label{SVRSECLEM005}
Let $E$ be an elliptic curve, let $n$ be a positive integer satisfying $4\mid n$ and $8\nmid n$ and let $\alpha_1\ne \alpha_2\in J(E)_{\mathrm{tors}}$ be two torsions of order $n$. If
$$
\begin{aligned}
\langle \alpha_1,\alpha_2 \rangle &= \dfrac{n}2 \hspace{12pt}\text{in } J(E)_n \text{ and}\\
4(d_1\alpha_1 - d_2\alpha_2) &= 0
\end{aligned}
$$
for some odd integers $d_1$ and $d_2$, then
$$
\ord(d_1\alpha_1 - d_2\alpha_2) = 2.
$$
\end{Lemma}

\begin{proof}
Let $m = n/2$. We may assume that
$$
\alpha_1 = \begin{bmatrix}
1\\
0
\end{bmatrix}
\hspace{12pt}\text{and}\hspace{12pt} \alpha_2 = \begin{bmatrix}
a\\
m
\end{bmatrix}
$$
in $J(E)_n\cong \ZZ/n\ZZ\times \ZZ/n\ZZ$, where $\gcd(a,m) = 1$ and hence $a$ is odd. Then
$$
d_1 \alpha_1 - d_2 \alpha_2 = \begin{bmatrix}
d_1 - ad_2\\
-d_2m
\end{bmatrix}
$$
and $2m \mid 4(d_1 -ad_2)$. And since $d_1 -ad_2$ is even and $4\nmid m$, we see that $2m\mid 2(d_1 - ad_2)$ and hence $d_1 \alpha_1 - d_2 \alpha_2$
has order $2$.
\end{proof}

Now we are ready to prove Propositions \ref{SVRSECPROP005} and \ref{SVRSECPROP006}.

\begin{proof}[Proof of Proposition \ref{SVRSECPROP005}]
Suppose that $E$ is a general fiber of a Bryan-Leung K3 surface $\pi: X\to B = \PP^1$ with $24$ nodal fibers. Let
$$
n = \operatorname{lcm}(n_1,n_2),\ d_1 = \dfrac{n}{n_1} \text{ and } d_2 = \dfrac{n}{n_2}.
$$
Suppose that
$$
\big\{b_{\tau_1,p}(z) = 0\big\} = \big\{q_1,q_2\big\} \hspace{12pt}\text{and}\hspace{12pt}
\big\{b_{\tau_2,p}(z) = 0\big\} = \big\{q_1,q_3\big\}.
$$
It suffices to prove that one of $p-q_1, p-q_2$ and $p-q_3$ is torsion.

Since $\ord(\tau_i) = n_i$, $\tau_i = d_i \alpha_i$ for $i=1,2$ and some $\alpha_i\in J(E)_\text{tors}$ of order $n$. Let $m = \langle \alpha_1, \alpha_2\rangle \in \ZZ/n\ZZ$.

By Lemma \ref{SVRSECLEM002},
$$
\begin{aligned}
b_{\tau_2 + kd_2 m \alpha_1, p}(q_1) &= 0 \hspace{24pt}\text{for } 2\mid k \text{ and }\\
b_{\tau_2 + kd_2 m \alpha_1, p}(q_2) &= 0 \hspace{24pt}\text{for } 2\nmid k
\end{aligned}
$$

If $k_0 = n/\gcd(d_2m,n)$ is odd, then $\tau_1 = \tau_2$ by Lemma \ref{SVRSECLEM002}, which is a contradiction. Therefore, $k_0$ and $n$ are even. If $k_0 \ne 2$, we have two cases:
\begin{itemize}
\item Suppose that $4\mid k_0$. We have
$$
b_{\tau_2,p}(q_1) = b_{\tau_2 + (k_0/2)d_2 m\alpha_1, p}(q_1) = 0.
$$
Let
$$
\tau_1' = \tau_2 + (k_0/2)d_2 m\alpha_1 \hspace{24pt}\text{and}\hspace{24pt} \tau_2' = \tau_2.
$$
Suppose that
$$
\begin{aligned}
\big\{b_{\tau_1',p}(z) = 0\big\} &= \big\{q_1,q_2'\big\} \hspace{12pt}\text{and}\\
\big\{b_{\tau_2',p}(z) = 0\big\} &= \big\{q_1,q_3'\big\}.
\end{aligned}
$$
By Lemma \ref{SVRSECLEM004},
$$
b_{\tau_1' - \tau_2',p}(q_2') = 0.
$$
Obviously, $\ord(\tau_1' - \tau_2') = 2$. Therefore, $p - q_2'\in J(E)_\text{tors}$ by Lemma \ref{SVRSECLEM003}. It follows that $p - q_1\in J(E)_\text{tors}$ and we are done.

\item Suppose that $4\nmid k_0$ and $k_0 > 2$. We have
$$
b_{\tau_2,p}(q_1) = b_{\tau_2 + 2d_2m\alpha_1, p}(q_1) = 0.
$$
Let
$$
\tau_1' = \tau_2 + 2d_2 m\alpha_1 \hspace{24pt}\text{and}\hspace{24pt} \tau_2' = \tau_2.
$$
We see that $\tau_1'\ne \tau_2'$,
$$
\ord(\tau_1') \mid n_2 = \ord(\tau_2')
$$
and
$$
\langle \tau_1', \tau_2' \rangle = m' = 2 (d_2 m)^2
$$
with $n_2/\gcd(m',n_2)$ odd. Then $\tau_1' = \tau_2'$ by Lemma \ref{SVRSECLEM002}, which is a contradiction.
\end{itemize}

So we have $k_0=2$. That is,
$$
n = 2 \gcd(d_2m,n).
$$
Similarly, we have
$$
n = 2 \gcd(d_1m,n).
$$
So we have
$$
d_2m \equiv d_1m \equiv \dfrac{n}2 \ (\text{mod } n).
$$
And since $\gcd(d_1,d_2) = 1$, we conclude that
$$
m \equiv \dfrac{n}2 \ (\text{mod } n)
$$
and $d_1$ and $d_2$ are both odd. That is, we have reduced the proposition to the case
that 
\begin{equation}\label{SVRSECE015}
2\mid n,\ 2\nmid d_1d_2\text{ and } m = \dfrac{n}2.
\end{equation}
Note that under these assumptions,
$$
m \tau_j = d_i m \alpha_j = m \alpha_j
$$
for all $i,j=1,2$.

If one of $\tau_i$ is a $2$-torsion, then it follows immediately from Lemma \ref{SVRSECLEM003} that $p - q_1\in J(E)_\text{tors}$ and we are done. So we may assume that $n_i \ge 3$ for $i=1,2$.

By Lemma \ref{SVRSECLEM004},
$$
b_{\tau_1 - \tau_2,p}(q_2) = 0.
$$
If $\tau_1 - \tau_2$ is a $2$-torsion, then $p - q_2\in J(E)_\text{tors}$ by Lemma \ref{SVRSECLEM003}. We are again done. So we may assume that none of $\tau_1$, $\tau_2$ and $\tau_1 - \tau_2$ are $2$-torsions. That is, we may assume that
\begin{equation}\label{SVRSECE027}
n_1 \ge 3,\ n_2 \ge 3 \text{ and } \ord(\tau_1 - \tau_2) \ge 3
\end{equation}
in addition to \eqref{SVRSECE015}.

Repeatedly applying Lemma \ref{SVRSECLEM002}, we obtain
$$
\begin{aligned}
\big\{ b_{\tau_1,p}(z) = 0\big\} &= \big\{q_1, q_2\big\}\\
\big\{ b_{\tau_2,p}(z) = 0\big\} &= \big\{ q_1, q_3\big\}\\
\big\{ b_{\tau_2 + m\alpha_1,p}(z) = 0 \big\} &= \big\{q_2, q_4\big\}\\
\big\{ b_{\tau_1 + m\alpha_2,p}(z) = 0 \big\} &= \big\{ q_3, q_5 \big\}
\end{aligned}
$$
Continuing this process, we obtain
$$
b_{\tau_1 + m(\alpha_2 + m\alpha_1),p}(q_4) = 0.
$$
Suppose that $4\mid n$, i.e., $2\mid m$. Then $m(\alpha_2 + m\alpha_1) = m\alpha_2$ and hence
$$
b_{\tau_1 + m\alpha_2,p}(q_4) = 0.
$$
Since $\{b_{\tau_1 + m\alpha_2,p}(z) = 0\} = \{q_3,q_5\}$, we have either $q_3 = q_4$ or $q_4 = q_5$.
\begin{itemize}
\item If $q_3 = q_4$, then
$$
\begin{aligned}
\big\{b_{\tau_1,p}(z) = 0\big\} &= \big\{q_1, q_2\big\}\\
\big\{b_{\tau_2,p}(z) = 0\big\} &= \big\{q_1, q_3 \big\}\\
\big\{b_{\tau_2 + m\alpha_1,p}(z) = 0\big\} &= \big\{q_2, q_3\big\}
\end{aligned}
$$
and hence
$$
\begin{aligned}
(p-q_1) + (p - q_2) &= \tau_1 \in J(E)_\text{tors}\\
(p-q_1) + (p - q_3) &= \tau_2 \in J(E)_\text{tors}\\
(p-q_2) + (p - q_3) &= \tau_2 + m\alpha_1 \in J(E)_\text{tors}
\end{aligned}
$$
It follows that $p-q_1, p-q_2, p-q_3\in J(E)_\text{tors}$. We are done.

\item If $q_4 = q_5$, then
\begin{equation}\label{SVRSECE016}
\begin{aligned}
\big\{b_{\tau_1,p}(z) = 0\big\} &= \big\{q_1, q_2\big\}\\
\big\{b_{\tau_2,p}(z) = 0\big\} &= \big\{q_1, q_3 \big\}\\
\big\{b_{\tau_2 + m\alpha_1,p}(z) = 0\big\} &= \big\{q_2, q_4\big\}\\
\big\{b_{\tau_1 + m\alpha_2,p}(z) = 0\big\} &= \big\{q_3, q_4\big\}
\end{aligned}
\end{equation}
and hence
$$
\begin{aligned}
(p-q_1) + (p - q_2) &= \tau_1\\
(p-q_1) + (p - q_3) &= \tau_2\\
(p-q_2) + (p - q_4) &= \tau_2 + m\alpha_1 = \tau_2 + m\tau_1\\
(p-q_3) + (p - q_4) &= \tau_1 + m\alpha_2 = \tau_1 + m\tau_2
\end{aligned}
$$
It follows that
$$
(m-2) (\tau_1 - \tau_2) = 0 \Rightarrow \gcd(m-2,n) (\tau_1 - \tau_2) = 0.
$$
Since $\gcd(m-2,n) = \gcd(m-2,2m)$ is either $2$ or $4$, the order of $\tau_1 - \tau_2$ is either $2$ or $4$. 
By our hypothesis \eqref{SVRSECE027}, $\ord(\tau_1 - \tau_2) \ne 2$.
So $\ord(\tau_1 - \tau_2) = 4$. Then $\gcd(m-2,2m) = 4$ and $4\nmid m$. This contradicts 
Lemma \ref{SVRSECLEM005}.
\end{itemize}

So far we have proved the proposition when $m$ is even. Suppose that $2\nmid m$. Then $m(\alpha_2 + m\alpha_1) = m(\alpha_1 + \alpha_2)$ and hence
$$
b_{\tau_1 + m(\alpha_1 + \alpha_2),p}(q_4) = 0.
$$
Continuing applying Lemma \ref{SVRSECLEM002}, we obtain
$$
\begin{aligned}
\big\{b_{\tau_1,p}(z) = 0\big\} &= \big\{q_1,q_2\big\}\\
\big\{b_{\tau_2,p}(z) = 0\big\} &= \big\{q_1,q_3\big\}\\
\big\{b_{\tau_2 + m\alpha_1,p}(z) = 0\big\} &= \big\{ q_2, q_4\big\}\\
\big\{b_{\tau_1 + m\alpha_2,p}(z) = 0\big\} &= \big\{ q_3, q_5\big\}\\
\big\{b_{\tau_1 + m(\alpha_1 + \alpha_2),p}(z) = 0 \big\} &= \big\{q_4, q_6\big\}\\
\big\{b_{\tau_2 + m(\alpha_1 + \alpha_2),p}(z) = 0 \big\} &= \big\{q_5, q_7\big\}
\end{aligned}
$$
Applying Lemma \ref{SVRSECLEM002} to $(\tau_1 + m(\alpha_1 + \alpha_2), \tau_2 + m\alpha_1)$, we obtain
$$
b_{\tau_2 + m(\alpha_1 + \alpha_2),p}(q_6) = 0.
$$
Similarly,
$$
b_{\tau_1 + m(\alpha_1 + \alpha_2),p}(q_7) = 0.
$$
That is, $q_6\in \big\{q_5, q_7\big\}$ and $q_7\in \big\{q_4, q_6\big\}$. Since $\big\{q_5, q_7\big\}\ne \big\{q_4, q_6\big\}$, we must have $q_6 = q_7$. Then from
\begin{equation}\label{SVRSECE017}
\begin{aligned}
\big\{b_{\tau_1,p}(z) = 0\big\} &= \big\{q_1,q_2\big\}\\
\big\{b_{\tau_2,p}(z) = 0\big\} &= \big\{q_1,q_3\big\}\\
\big\{b_{\tau_2 + m\alpha_1,p}(z) = 0\big\} &= \big\{ q_2, q_4\big\}\\
\big\{b_{\tau_1 + m\alpha_2,p}(z) = 0\big\} &= \big\{ q_3, q_5\big\}\\
\big\{b_{\tau_1 + m(\alpha_1 + \alpha_2),p}(z) = 0 \big\} &= \big\{q_4, q_6\big\}\\
\big\{b_{\tau_2 + m(\alpha_1 + \alpha_2),p}(z) = 0 \big\} &= \big\{q_5, q_6\big\}
\end{aligned}
\end{equation}
we obtain
$$
3(\tau_1 - \tau_2) = m(\alpha_1 - \alpha_2).
$$
Hence $\tau_1 - \tau_2$ has order $2$ or $6$. 

By our hypothesis \eqref{SVRSECE027}, $\ord(\tau_1 - \tau_2) \ne 2$. So
$\tau_1 - \tau_2$ has order $6$. Hence $6\mid n$, $3\mid m$ and $3\mid n_1n_2$.

Since $d_1$ and $d_2$ are odd, $n_1 = n/d_1$ and $n_2=n/d_2$ are even. So at least one of $n_1$ and $n_2$ is divisible by $6$. Without the loss of generality, let us assume that $6\mid n_1$. Then
$$
n_1(\tau_1 - \tau_2) = 0 \Rightarrow n_1\tau_2 = 0 \Rightarrow n_2\mid n_1\Rightarrow n = n_1.
$$
Let
$$
\tau_1' = \tau_1 \hspace{24pt}\text{and}\hspace{24pt} \tau_2' = \tau_1 - \tau_2.
$$
By Lemma \ref{SVRSECLEM004},
$$
b_{\tau_1',p}(q_2) = b_{\tau_2',p}(q_2) = 0.
$$
Applying the whole argument to $(\tau_1',\tau_2')$, we again arrive at
$$
\ord(\tau_1' - \tau_2') = 6.
$$
That is, $n_2 = \ord(\tau_2) = 6$. Then this implies that $\tau_1 = \tau_2 + (\tau_1 - \tau_2)$ also has order $6$. So we have \eqref{SVRSECE021}.
\end{proof}

\begin{proof}[Proof of Proposition \ref{SVRSECPROP006}]
Let
$$
\begin{aligned}
\big\{b_{\tau_1,p}(z) = 0\big\} &= \big\{q_1, q_2\big\} \hspace{24pt}\text{and}\\
\big\{b_{\tau_2,p}(z) = 0\big\} &= \big\{q_1, q_3\big\}
\end{aligned}
$$
where $\eta_i = p - q_i$ are torsions for $i=1,2,3$.

Suppose that $n = \operatorname{lcm}(\ord(\tau_1), \ord(\eta_1), \ord(\eta_2))$ and $\tau_1 = d \alpha_1$ for some $\alpha_1\in J(E)_\text{tors}$ of order $n$. Let $E$ be a general fiber of a Bryan-Leung K3 surface $X\to B=\PP^1$ with $24$ nodal fibers. Clearly, each $\lambda\in \pi_1(B\backslash (S\cup S_n))$ acts on $(\tau_1, q_1, q_2)$ by
$$
\lambda(\tau_1, q_1,q_2) = (\lambda(\tau_1), p - \lambda(\eta_1), p - \lambda(\eta_2)).
$$
On the other hand, for $\lambda(\eta) = \eta + \langle \eta, \alpha_1\rangle \alpha_1$,
$$
\lambda(\tau_1, q_1,q_2) = (\lambda(\tau_1), q_2, q_1)
$$
by Lemma \ref{SVRSECLEM001}. Therefore,
$$
\eta_2 = \lambda(\eta_1) = \eta_1 + m \alpha_1
$$
for $m = \langle \eta_1, \alpha_1\rangle$. And since $\tau_1 = \eta_1 + \eta_2$, we have
$$
\tau_1 = 2\eta_1 + m\alpha_1 \Rightarrow \langle (d-m)\alpha_1,\alpha_1\rangle = \langle 2\eta_1, \alpha_1\rangle \Rightarrow 2m = 0
$$
in $\ZZ/n\ZZ$. If $m=0$, then $\eta_1 = \eta_2$, which contradicts Proposition \ref{SVRSECPROP000}. So $n$ is even and $m = n/2$. Therefore, we have
\begin{equation}\label{SVRSECE022}
\ord(\eta_1 - \eta_2) = \ord(\tau_1 - 2\eta_1) = 2.
\end{equation}
Similarly,
\begin{equation}\label{SVRSECE023}
\ord(\eta_1 - \eta_3) = \ord(\tau_2 - 2\eta_1) = 2.
\end{equation}
It follows that $\tau_1 - \tau_2$ is a $2$-torsion as well. By Lemma \ref{SVRSECLEM004},
$$
b_{\tau_1-\tau_2,p}(q_2) = 0.
$$
Hence $\eta_2$ is a $2$-torsion by Lemma \ref{SVRSECLEM005}. 
Together with \eqref{SVRSECE022} and \eqref{SVRSECE023}, we see that all of $\tau_1, \tau_2, \eta_1,\eta_2,\eta_3$ are $2$-torsions.
\end{proof}

To finish the proof of Proposition \eqref{SVRSECPROP001}, it remains to justify \eqref{SVRSECE026}.

\begin{proof}[Proof of Proposition \eqref{SVRSECPROP001}]
We have proved \eqref{SVRSECE024} with two exceptions outlined in the proposition.

If \eqref{SVRSECE026} fails, we must have one of the following:
\begin{enumerate}
\item[A.] $\tau_1, \tau_2, \tau_3$ are three distinct $2$-torsions.
\item[B.] $\tau_1, \tau_2, \tau_3$ are three distinct $6$-torsions satisfying that $\langle \tau_i, \tau_j\rangle = 3$ and $\ord(\tau_i - \tau_j) = 6$ for all $1\le i < j\le 3$.
\end{enumerate}

In case A, by \eqref{SVRSECE025},
$$
\begin{aligned}
\{b_{\tau_1,p}(z) = 0\} \cap \{b_{\tau_2,p}(z) = 0\} &= \{ p - \tau_3 \}\\
\{b_{\tau_2,p}(z) = 0\} \cap \{b_{\tau_3,p}(z) = 0\} &= \{ p - \tau_1 \}
\end{aligned}
$$
and \eqref{SVRSECE026} follows.

In case B, we may assume that
$$
\tau_1 = \begin{bmatrix}
1\\
0
\end{bmatrix}\in J(E)_6\cong \ZZ/6\ZZ\times \ZZ/6\ZZ
$$
Since $\langle \tau_i, \tau_j\rangle = 3$ for $i\ne j$, we must have
$$
\tau_2 = \begin{bmatrix}
a\\
3
\end{bmatrix} \text{ and } \tau_3 = \begin{bmatrix}
b\\
3
\end{bmatrix}
$$
for some $a, b\in \ZZ$ satisfying $3\nmid ab$ and $2\nmid (a-b)$. 

Since $\ord(\tau_1 - \tau_2) = \ord(\tau_1 - \tau_3) = 6$, $3\nmid (a-1)(b-1)$. Together with $3\nmid ab$, we must have
$$
a \equiv b \equiv 2\ (\text{mod } 3).
$$
Then $\ord(\tau_2 - \tau_3) = 2$. Therefore, there are no such triples $(\tau_1,\tau_2,\tau_3)$.
\end{proof}

\section{Proof of Theorem \ref{K3RATNOTESTHM808} for $g\ge 2$}

It remains to prove Theorem \ref{K3RATNOTESTHM808} for $g\ge 2$. As mentioned in Section \ref{SVRSECSECINTRO}, we will reduce it to the case $g=1$ by a degeneration argument.

Let $E$ be a smooth elliptic curve. We first construct a smooth projective family $X$ of surfaces  over $\Delta = \BA^1$ such that $X_0 \cong E\times \PP^1$ and $X_t\cong \PP \calE$ for $t\ne 0$, where $\calE$ is the rank $2$ vector bundle on $E$ given by a nonzero vector in $\Ext(\OO_E, \OO_E)$.

Let $\CV$ be a rank $2$ vector bundle over $E\times\Delta$ given by
$$
t\in \Ext(\OO_{E\times\Delta}, \OO_{E\times\Delta}) = H^1(\OO_{E\times\Delta}) = \CC[t]
$$
and let $X = \PP \CV$. Clearly, $X$ is such a family.

There is an effective divisor $D\subset X$, flat over $\Delta$, such that $D_t$ is the section of $X_t/E$ with $D_t^2 = 0$. Fix a point $p\in E$ and let $L = m D + \pi^* p$, where $\pi$ is the projection $X\to E$.

For $t\ne 0$, the Severi variety $V_{X_t,L,g}$ has expected dimension $g$. If we fix $g$ general points on $X_t$, there exist finitely many $[C]\in V_{X_t,L,g}$ such that $C$ passes through these points. Let us fix $g$ general sections $P_1,P_2,...,P_g \subset X$ of $X/\Delta$. Then after a base change, there exists a flat projective family $C\subset X$ of curves over $\Delta$ such that $C_t$ is an integral curve in $|L|$ on $X_t$ passing through $P_i\cap X_t$ for $i=1,2,...,g$ and $t\ne 0$.
Here we replace $\Delta$ by an analytic disk or a smooth affine curve finite over $\BA^1$.

Furthermore, we may choose the base change such that there exists a family of stable maps $\varphi: \calC\to X$ over $\Delta$ such that $\varphi$ maps $\calC$ birationally onto $C$.

On $X_0$, the linear system $|L|$ is completely reducible in the sense that
$$
H^0(\OO_{X_0}(L)) = \operatorname{Sym}^m H^0(\OO_{X_0}(D)) \otimes H^0(\OO_{X_0}(\pi^* p)).
$$
Therefore,
$$
C_0 = m_1 D_1 + m_2 D_2 + ... + m_g D_g + F
$$
where $D_i$ are the sections of $X_0/E$ passing through $P_i\cap X_0$ for $i=1,2,...,g$, $F$ is the fiber of $\pi: X_0\to E$ over $p$ and $m_i$ are positive integers such that $\sum m_i = m$.

Clearly, $C_t$ has only singularities in open neighborhoods of $D_i$. So it suffices to show that $C_t$ has only nodes and ordinary triple points as singularities in an analytic neighborhood of each $D_i$ for $i=1,2,...,g$, if $E$ is general.

Since $\calC_t$ is a smooth projective curve of genus $g$ for $t\ne 0$, there are exactly $g$ irreducible components $\Gamma_1, \Gamma_2, ..., \Gamma_g$ of $\calC_0$ such that each $\Gamma_i$ is a smooth elliptic curve dominating $D_i$ for $i=1,2,...,g$.

Let us fix $i$. If $m_i = 1$, there is nothing to do. Otherwise, Suppose that $m_i\ge 2$.
Let $\psi: \widehat{X}\to X$ be the blowup of $X$ along $D_i$. Then the central fiber $\widehat{X}_0 = S\cup R$ is a union of two smooth projective surfaces $S$ and $R$, where $S$ is the proper transform of $X_0$, $R$ is the exceptional divisor of $\psi$ and $S$ and $R$ meet transversely along a curve over $D_i$, which we still denote by $D_i$. Let $\widehat{C}$ be the proper transform of $C$ under $\psi$.

The rational map $\psi^{-1}\circ \varphi: \calC\dashrightarrow \widehat{X}$ is regular at a general point of $\Gamma_i$. We claim that
$$
\psi^{-1}\circ \varphi(\Gamma_i) \not \subset D_i = S\cap R.
$$
Otherwise, we choose a local section $Q$ of $\calC/\Delta$ passing through a general point of $\Gamma_i$. Then $\varphi(Q)$ is a local section of $\widehat{X}/\Delta$ meeting $D_i = S\cap R$, which is impossible since $\widehat{X}$ is smooth. So $\psi^{-1}\circ \varphi$ maps $\Gamma_i$ to an irreducible curve on $R$ other than $D_i$. That is, $\widehat{C}_0$ does not contain $D_i$.

We have either $R\cong \PP\calE$ or $R\cong E\times \PP^1$. 

\begin{enumerate}
\item[A.] If $R\cong \PP\calE$, then $\widehat{C}\cap R$ must be an integral curve in $|m_i\widehat{D} + \widehat{\pi}^* p|$ of geometric genus $1$, where $\widehat{D}$ is the proper transform of $D$ and $\widehat{\pi} = \pi\circ \psi$ is the projection $\widehat{X}\to E$. Then by Theorem \ref{K3RATNOTESTHM809}, $\widehat{C}\cap R$ has only nodes and ordinary triple points as singularities and the same holds for $C_t$ in an open neighborhood of $D_i$.

\item[B.] If $R\cong E\times \PP^1$, then $\widehat{C}\cap R = m_i \widehat{D}_i + \widehat{F}$, where $\widehat{D}_i$ is the section $R/E$ passing through the point $\widehat{P}_i\cap R$ with $\widehat{P}_i$ being the proper transform of $P_i$ under $\psi$ and $\widehat{F}$ is the fiber of $R$ over $p\in E$. So we continue to blow up $\widehat{X}$ along $\widehat{D}_i$. By embedded resolution of singularities, there exists a sequence blowups over $D_i$, say $f: X'\to X$, such that the proper transform $C'$ of $C$ is smooth over a general point of $D_i$. Then by Zariski's main theorem, the map $f^{-1}\circ \varphi: \calC\dashrightarrow X'$ has connected fiber over $f^{-1}(D_i)$. This means that $C_0'$ is smooth over a general point of $D_i$. So we will eventually end up in case A after a sequence of blowups over $D_i$. 
\end{enumerate}


\end{document}